\documentclass{amsart}
\usepackage{graphicx} 
\usepackage{amsmath,amssymb,amsthm,stmaryrd,latexsym,amsfonts,mathtools,tikz,mathrsfs} 
\usepackage[top=30truemm,bottom=30truemm,left=25truemm,right=25truemm]{geometry}

\newtheorem{thm}{Theorem}[section]
\newtheorem{lemm}[thm]{Lemma}

\newtheorem{rem}[thm]{Remark}

\theoremstyle{definition}

\makeatletter
\@addtoreset{equation}{section}

\makeatother

\def\Wick#1{\mathord{{:}{#1}{:}}}

\allowdisplaybreaks

\title{Stochastic quantization of the weighted exponential QFT}

\author{Seiichiro Kusuoka}
\address{Graduate School of Science, Kyoto University, Kitashirakawa-Oiwakecho, Sakyo-ku, Kyoto, 606-8502, Japan}
\email{kusuoka@math.kyoto-u.ac.jp}

\author{Hirotatsu Nagoji}
\address{Graduate School of Advanced Science and Engineering, Hiroshima University, 1-3-1 Kagamiyama, Higashi-Hiroshima City, Hiroshima, 739-8526, Japan}
\email{nagojihiro@hiroshima-u.ac.jp}

\date{}


\begin{document}

\begin{abstract}
We consider the stochastic quantization equation associated with the weighted exponential quantum field model (or the H\o egh-Krohn model) on the two dimensional torus. Unlike in the case of the usual (unweighted) exponential model, the drift term of the stochastic quantization equation can be both positive and negative, and that makes the equation more difficult to treat. We prove the unique existence of the time-global solution under a certain initial condition by a pathwise PDE argument in the so-called $L^2$-regime. We also see that this solution is properly associated with a Dirichlet form canonically constructed from the weighted exponential quantum field measure. 

\end{abstract}

\maketitle

\section{Introduction}
We consider the equation 
\begin{align}\label{eq1}
\partial_t \Phi_t = \frac{1}{2} (\Delta - 1)\Phi_t - \frac{1}{2}\int_{[-\alpha_0, \alpha_0]} \alpha \exp^{\Diamond}{(\alpha \Phi_t)}\nu(d\alpha) + \dot{W}_t
\end{align}
on the 2-dimensional torus $\Lambda \coloneqq \mathbb{T}^2 = (\mathbb{R}/2\pi\mathbb{Z})^2$, where $\nu$ is a finite nonnegative Borel measure on $[-\alpha_0,\alpha_0]$, $\alpha_0\in (0,\sqrt{8\pi})$ and $\dot{W}$ is a space time white noise or equivalently, $\{W_t\}_{t\ge 0}$ is a cylindrical Brownian motion on $L^2(\Lambda)$.
This is the stochastic quantization equation of the weighted exponential quantum field model, and its invariant probability measure is given by
\begin{align}
\mu^{(\nu)}(d\phi) = \frac{1}{Z^{(\nu)}}\exp{\left( - \int_{[-\alpha_0,\alpha_0]}  \left(\int_\Lambda \exp^{\Diamond}{(\alpha \phi)(x)}dx\right) \nu(d\alpha)    \right)} \mu_0(d\phi),\label{eq:measure}
\end{align}
where $\mu_0$ is the massive Gaussian free field defined as a Gaussian probability measure on the space of periodic distributions $\mathcal{D}'(\Lambda)$ with covariance $(1-\Delta)^{-1}$, $Z^{(\nu)}$ is the normalization constant and $\exp^{\Diamond}(\alpha \phi)$ denotes the Wick exponential formally given by
\begin{equation}\label{eq:wickexp}  \exp^{\Diamond}(\alpha \phi)(x) = \exp{\left(\alpha \phi(x) - \frac{\alpha^2}{2}\int |\phi^2(x)|\mu_0(d\phi)  \right)}.  
\end{equation}
Note that the integral in \eqref{eq:wickexp} is ill-defined because $\phi$ is a genuine distribution, not a function $\mu_0$-almost surely so this expression is only formal. However, it is known that if $\alpha^2<8\pi$, we can make sense of it by way of a renormalization procedure, see Section \ref{sec:pre} for the precise definition. When $\nu$ is a Dirac measure $\delta_\alpha$, the model is called the $\exp(\Phi)_2$-model or the H\o egh-Krohn model and known as one of the simplest rigorous (Euclidean) quantum field models. Inspired by the recent progress in the research of singular stochastic PDEs, there have been some works that study the stochastic quantization of the $\exp(\Phi)_2$-model by using PDE techniques such as \cite{G20,ORW21,HKK21a,HKK23}. Especially in \cite{HKK23}, the authors constructed the $\exp(\Phi)_2$-measure on the torus and proved the unique existence of the strong solution to the stochastic quantization equation for any $\alpha^2<8\pi$. See \cite{S74} or the references in the above-mentioned papers for the physical background of the model.
The weighted $\exp(\Phi)_2$-model, which is also called H\o egh-Krohn model, has also been studied in the physics literature (e.g. \cite{H71,AH74,FP77,AL04}) with a particular interest in the $\cosh(\Phi)_2$-model (or the sinh-Gordon model), which is the case $\nu =(\delta_{\alpha} + \delta_{-\alpha})/2$, see e.g. \cite{Pil98, Zam06}. In the case of the weighted model, the drift term in the stochastic quantization equation \eqref{eq1} can be both positive and negative unless we have $\mbox{supp}(\nu) \subset [0,\alpha_0]$ or $\mbox{supp}(\nu) \subset [-\alpha_0,0]$, and that makes some techniques usually used in the case of the unweighted models unavailable. Due to some difficulties like this, there have been less studies on the stochastic quantization of the weighted $\exp(\Phi)_2$-model. In \cite{BDV25}, the elliptic stochastic quantization of the infinite-volume sinh-Gordon model for $\alpha^2<4\pi$ (the so-called $L^2$-regime) is considered and the authors proved that the equation is well-posed and the sinh-Gordon quantum field measure satisfies the Osterwalder-Schrader axioms. See also \cite[Remark 1.18]{ORW21}, in which the authors constructed the solution to the parabolic stochastic quantization equation of the sinh-Gordon model on the torus for $\alpha^2< 8\pi/(3+2\sqrt2)$.

In this paper, we refine the argument in \cite{HKK21b} and study the stochastic quantization equation associated to the general weighted $\exp{(\Phi)}_2$-model on the torus for any finite nonnegative Borel measure $\nu$ on $[-\alpha_0,\alpha_0]$.
Our aim is to prove global well-posedness in the $L^2$-regime $\alpha_0^2 < 4\pi$ by a pathwise PDE argument. Then, we also construct a diffusion associated to the model in the $L^1$-regime $\alpha_0^2<8\pi$ by applying the general theory of Dirichlet form for SPDE developed in \cite{AR91} and see the connection between the two approaches as in \cite{HKK21a,HKK23}.
More precisely, we prove the following results. 
Since the construction of the Wick exponential $\exp^\Diamond(\alpha\cdot)$ involves renormalization, we introduce an approximation operator $P_N$, $N \in \mathbb{N}$ as follows.
   \begin{equation}\label{eq:appro}
       P_N f \coloneqq \sum_{l\in \mathbb{Z}^2, |l|\le A^N} \hat{f}(l) e_l,
   \end{equation}
   where $A>0$ will be chosen to be sufficiently large with dependence on $\alpha_0$ later in the proof of Lemma \ref{lem:stoch}, and we define the Fourier transformation of $f\in \mathcal{D}(\Lambda) \coloneqq C^\infty(\Lambda)$ by 
    \begin{equation}
        \hat{f}(l) \coloneqq \int_\Lambda f(x)e_l(x)dx,\ \ \ l\in\mathbb{Z}^2,
    \end{equation}
   where $e_l(x) \coloneqq \frac{1}{ 2\pi}e^{i l\cdot x}$. Let $H^s(\Lambda)$ ($s\in \mathbb{R}$) be the Sobolev spaces introduced in Section \ref{sec:not}.
\begin{thm}\label{thm1}
Let $\alpha_0^2<4\pi$, $\beta \in (\frac{\alpha_0^2}{4\pi},1)$ and let $\Phi^N$ be the solution to the equation
\begin{gather}
\begin{cases}
\partial_t \Phi^N_t = \frac{1}{2} (\Delta - 1)\Phi^N_t - \frac{1}{2} \int_{[-\alpha_0, \alpha_0]}\alpha\exp{(\alpha \Phi^N_t - \frac{\alpha^2}2C_N)}\nu(d\alpha) + P_N\dot{W}_t \label{eq1'} \\
\Phi^N_0 = P_N \xi + \eta, 
\end{cases}
\end{gather}
where $\xi$ is distributed according to the massive Gaussian free field measure $\mu_0$ with covariance $(1-\Delta )^{-1}$, and independent of $\{W_t\}_{t\ge 0}$, $C_N$ is the renormalization constant defined in \eqref{eq:const} and $\eta \in H^{2-\beta}(\Lambda)$ is non-random. Then, $\{\Phi^N\}_{N\in \mathbb{N}}$ converges to some $\Phi$ in $C([0,T];H^{-\beta}(\Lambda))$ almost surely as $N\rightarrow \infty$.
\end{thm}
\begin{rem}\label{rem:thm1}
In view of the fact that $\mu^{(\nu)}$ is absolutely continuous with respect to $\mu_0$ (see Section \ref{sec:dirichlet}), Theorem~\ref{thm1} especially implies that the solution $\Phi^N(\phi)$ of
    \begin{align*}
\begin{cases}
\partial_t \Phi^N_t = \frac{1}{2} (\Delta - 1)\Phi^N_t - \frac{1}{2} \int_{[-\alpha_0, \alpha_0]}\alpha\exp{(\alpha \Phi^N_t - \frac{\alpha^2}2C_N)}\nu(d\alpha) + P_N\dot{W}_t \label{eq1} \\
\Phi^N_0 = P_N \phi, 
\end{cases}
\end{align*}    
converges to some $\Phi^N(\phi)$ in $C([0,T];H^{-\beta}(\Lambda))$ almost surely as $N\rightarrow \infty$ for $\mu^{(\nu)}$-almost every $\phi$.
\end{rem}

\begin{rem}
The limit $\Phi$ in Theorem \ref{thm1} can be decomposed by $\Phi= X + Y$, where $X,Y$ are the unique mild solutions to \eqref{eq:defx} and \eqref{eqY2}, respectively, see Remark \ref{rem:lemconv}. 
    
\end{rem}

\begin{rem}\label{rem:thm1-2}
If we have $\mbox{supp}(\nu) \subset [-\alpha_0, 0]$ or $\mbox{supp}(\nu)\subset [0,\alpha_0]$, many of the known results in the case of the unweighted model (which is the case $\nu =\delta_\alpha$) should be extended to our setting in a straightforward way. For example, all the results in \cite{HKK23} can be extended to our setting and in particular, the convergence of $\Phi^N$ in Theorem \ref{thm1} holds in the full $L^1$-regime $\alpha_0^2<8\pi$ in that case, see Remark \ref{rem:estey2}.   
\end{rem}

\begin{thm}\label{thm2}
   Let $\alpha_0^2<8\pi$, $\beta\in (0,1)$ and $E \coloneqq H^{-\beta}(\Lambda)$. Let $\mathcal{E}$ be a pre-Dirichlet form defined by
    \begin{equation}
        \mathcal{E}\left(F, G\right) \coloneqq \frac{1}2 \int_{E} \left\langle \nabla F (\phi) , \nabla G(\phi) \right\rangle_{L^2(\Lambda)} \mu^{(\nu)}(d\phi),\ \ F,G\in \mathfrak{F}C^\infty_b \subset L^2(\mu^{(\nu)}),
    \end{equation}
    where $\mathfrak{F}C^\infty_b$ is the space of all bounded smooth cylindrical functionals and $\nabla$ denotes the $ L^2(\Lambda)$-derivative, see Section \ref{sec:dirichlet} for the more precise definitions. Then, there exists a suitable $\beta \in (0,1)$ such that $(\mathcal{E},\mathfrak{F}C_b^\infty)$ is closable in $L^2(E,\mu^{(\nu)})$ and its closure $(\bar{\mathcal{E}},\mathcal{D}(\bar{\mathcal{E}}))$ is a quasi-regular Dirichlet form. Moreover, if $\alpha_0^2<4\pi$ and $\beta\in (\frac{\alpha_0^2}{4\pi},1)$, $\bar{\mathcal{E}}$ is properly associated with the process $\Phi$ constructed in Remark \ref{rem:thm1}.
\end{thm}
\begin{rem}
    We say that a process $\Phi(\phi),\ \phi\in E$ is properly associated with a Dirichlet form $\bar{\mathcal{E}}$ when the Markov semigroup $P_t$ on $L^2(\mu^{(\nu)})$ corresponding to $\bar{\mathcal{E}}$ is given by $P_t F (\phi) = \mathbb{E}[F(\Phi_t(\phi)) ]$ for any bounded measurable function $F$ on $E$. Note that any quasi-regular Dirichlet form admits a diffusion process properly associated to it, see \cite[Theorem 3.5]{MR92}, for example.
\end{rem}
The organization of the present paper is as follows.
In Section \ref{sec:pre}, we recall some facts on the Sobolev spaces, heat kernel and Wick exponentials that are used in later sections. In Section \ref{sec:gwp} we prove Theorem~\ref{thm1}, where the proofs of some stochastic estimates are postponed to Section \ref{sec:stoch}. Finally, the proof of Theorem~\ref{thm2} is explained in Section \ref{sec:dirichlet}.

\subsection{Notations}\label{sec:not}
We use the following notations in this paper.
\begin{itemize}
    \item 
    We write $a \lesssim b$ when $a\le C b$ holds for some $C>0$ independent of the variables under consideration. When we want to emphasize the dependency of $C$ on the variable $x$, we write $a \lesssim_x b$. We also write $a \simeq b$ when we have $a \lesssim b$ and $a \gtrsim b$.
    \item 
    We write $\Lambda \coloneqq (\mathbb{R}/2\pi\mathbb{Z})^2$ and $\mathcal{D}(\Lambda) \coloneqq C^\infty(\Lambda)$. Let $\mathcal{D}'(\Lambda)$ be the topological dual of $\mathcal{D}(\Lambda)$.
   \item 
   Let $H^s (\Lambda)$ ($s\in \mathbb{R}$) 
   be the Sobolev spaces, see Section \ref{sec:besov} for the definition.
   \item 
We often use shorthand notations such as $L_T^pH^{s} \coloneqq L^p([0,T];H^s(\Lambda))$ and $L_{T,\alpha_0}^pH^s \coloneqq L^p([0,T]\times [-\alpha_0,\alpha_0],\lambda_T\otimes\nu;H^s(\Lambda))$, where $\lambda_T$ is the Lebesgue measure on $[0,T]$.
   \item The approximation operator $P_N$ ($N \in \mathbb{N}$) is defined by \eqref{eq:appro}.
\item
For $k\in\mathbb{N}$ and $c\in\mathbb{R}$, we define the $k$-th Hermite polynomial $H_k (x;c)$ by
\begin{equation}\label{herm}
e^{tx-\frac{1}{2}ct^2} = \sum_{k=0}^\infty \frac{t^k}{k!} H_k(x;c),\ \ \ t,x\in\mathbb{R}. 
\end{equation}   
\end{itemize}
\section{Preliminaries}\label{sec:pre}

\subsection{Sobolev spaces and heat kernel}\label{sec:besov}
Sobolev spaces $H^s(\Lambda)$ ($s\in \mathbb{R}$) are defined by the completion of the space $\mathcal{D}(\Lambda) = C^\infty(\Lambda)$ under the norm
\begin{equation}
    \| f \|_{H^{s}(\Lambda)} \coloneqq \| (1-\Delta)^{\frac{s}{2}}f\|_{L^2(\Lambda)} = \left(\int_\Lambda  |(1-\Delta)^{\frac{s}{2}}f(x)|^2     dx \right)^{\frac{1}2}, \ \ \ f\in \mathcal{D}(\Lambda). 
\end{equation}

We need to prepare the following lemmas for Section \ref{sec:gwp}. The proofs can be found in \cite[Theorem 3.4 and Proposition A.3]{HKK21a}.
\begin{lemm}\label{lem:prod}
    Let $s>0$. For any $f \in \mathcal{D}(\Lambda)$ and non-negative $g\in \mathcal{D}_+(\Lambda)$, it holds that 
    \[ \| fg \|_{H^{-s}(\Lambda)} \lesssim \| f \|_{C(\Lambda)} \| g \|_{H^{-s}(\Lambda)}  .\]
    In particular, when we define $H^{-s}_+(\Lambda)$ as the set of all non-negative distributions belonging to $H^{-s}(\Lambda)$, the product $fg$ for $f\in C(\Lambda)$ and $g\in H_+^{-s}(\Lambda)$ is a well-defined distribution by the unique continuous extension of the map 
    \[ C(\Lambda)\times H_+^{-s}(\Lambda) \supset  \mathcal{D}(\Lambda) \times \mathcal{D}_+(\Lambda) \ni (f,g) \mapsto fg \in H^{-s}(\Lambda).  \]
\end{lemm}

\begin{lemm}\label{lem:schauder}
    Let $u$ be the mild solution in $\mathcal{D}'(\Lambda)$ to
    \begin{align*}
\begin{cases}
\partial_t u = \frac{1}{2} (\Delta - 1)u  + U  \\
u_0 = v.
\end{cases}
    \end{align*}
    Then, for any $r\in (1,\infty]$, $\beta \in (0,1)$ and $\delta \in (0,1-\beta)$, we have
  \[\|u\|_{L^r([0,T];H^{1+\delta}(\Lambda))\cap C([0,T];H^{1 + \delta - \frac{2}{r}}(\Lambda))\cap C^{\delta/2}([0,T];H^{1-\frac{2}{r}}(\Lambda))} \lesssim \|v\|_{H^{1+\delta}(\Lambda)} + \| U\|_{L^r([0,T];H^{-\beta}(\Lambda))}   .\]  
\end{lemm}

\subsection{Ornstein-Uhlenbeck process and Wick exponentials}
We consider the infinite dimensional stationary Ornstein-Uhlenbeck process
\begin{equation}\label{eq:defx}
    X_t \coloneqq e^{\frac{t}{2}(\Delta-1)}\xi + \int^t_0 e^{\frac{t-s}{2}(\Delta -1)}dW(s),
\end{equation}
where $\xi$ is distributed according to the massive Gaussian free field $\mu_0$ with covariance $(1-\Delta )$, and independent of the $L^2(\Lambda)$-cylindrical Brownian motion $\{W_t\}_{t\ge0}$. Note that $X^N_t \coloneqq P_N X_t$ solves the linear equation 
\begin{gather}
\begin{cases}
\partial_t X^N_t = \frac{1}{2} (\Delta - 1)X^N_t  + P_N\dot{W}_t \label{eqX} \\
X^N_0 = P_N \xi, 
\end{cases}
\end{gather}
where $P_N$ is defined by \eqref{eq:appro}. We introduce the approximation of the Wick exponential defined by 
\[ \mathcal{D}'(\Lambda) \ni \phi \mapsto \exp_N^{\Diamond}(\alpha\phi) \coloneqq \exp\left(\alpha P_N \phi - \frac{\alpha^2}2 C_N \right) \in \mathcal{D}'(\Lambda)     ,\]
where 
\begin{equation}\label{eq:const} C_N \coloneqq \int_{\mathcal{D}'(\Lambda)} | P_N \phi (x)|^2 \mu_0(d\phi) = \langle (1-\Delta)^{-1}P_N \delta_x,P_N\delta_x\rangle = \frac{1}{(2\pi)^2}\sum_{l\in\mathbb{Z}^2,|l|\le A^N}\frac{1}{1+|l|^2} \simeq N     
\end{equation}
is the renormalization constant. 
The following results on the Ornstein-Uhlenbeck process and the Wick exponentials are known, see \cite[Theorem 2.1, Proposition 3.1 and Theorem 3.2]{HKK23}, for example. In the following lemma, let $B_{p,q}^s(\Lambda)$ ($s\in \mathbb{R}$, $p,q\in [1,\infty]$) denote the Besov spaces. Recall that $B_{2,2}^s(\Lambda)=H^s(\Lambda)$ with equivalent norms.
\begin{lemm}\label{lem:ou-wick}
    \begin{enumerate}
        \item 
        One has $X \in C([0,T];H^{-\epsilon}(\Lambda))$ almost surely for any $\epsilon>0$.
        \item 
        Let $\alpha^2 < 8\pi$, $p> 1$ and $s\in(\frac{\alpha^2}{4\pi}(p-1),1)$. Then, $\exp_N^{\Diamond}(\alpha \phi)$ converges to some $\exp^{\Diamond}(\alpha \phi)$ in $B_{p,p}^{-s}(\Lambda)$ for $\mu_0$-almost every $\phi$.
        \item 
      Let $\alpha^2 < 8\pi$, $p> 1$ and $s\in(\frac{\alpha^2}{4\pi}(p-1),1)$. Then,  $\exp_N^{\Diamond}(\alpha X_t)$ converges to some $\exp^{\Diamond}(\alpha X_t)$ in $L^p([0,T];B_{p,p}^{-s}(\Lambda))$ almost surely.
        \end{enumerate}
\end{lemm}

\section{Global well-posedness in the $L^2$-regime}\label{sec:gwp}
In this section, we prove Theorem \ref{thm1}.
We decompose the solution of the approximation equation \eqref{eq1'} by $\Phi^N = X^N + Y^N$, where $X^N$ solves the equation \eqref{eqX} and $Y^N$ solves 
\begin{gather}
\begin{cases}
\partial_t Y^N_t = \frac{1}{2} (\Delta - 1)Y^N_t - \frac{1}{2}\int_{[-\alpha_0,\alpha_0]}\alpha\exp{(\alpha (X^N_t + Y^N_t) - \frac{\alpha^2}2C_N)}\nu(d\alpha)  \label{eqY} \\
Y^N_0 = \eta. 
\end{cases}
\end{gather}
We define $\mathcal{X}_t^{\alpha,N}\coloneqq \exp_N^\Diamond(\alpha X_t^N)\coloneqq \exp{( \alpha X_t^N -\frac{\alpha^2}{2}C_N)}$.
Then, we have the following result, see Section \ref{sec:stoch} for the proof.
\begin{lemm}\label{lem:weightedexp}
Let $\alpha_0^2<4\pi$. For any $\beta >\frac{\alpha_0^2}{4\pi}$, $\left\{\mathcal{X}^{\alpha,N} \right\}_N$ converges to some $\mathcal{X}^\alpha$ in $L^2([0,T]\times [-\alpha_0, \alpha_0], \lambda_T \otimes \nu;H^{-\beta}(\Lambda))$ almost surely as $N\rightarrow \infty$, where $\lambda_T$ is the Lebesgue measure on $[0,T]$. 
\end{lemm}
In view of Lemma \ref{lem:weightedexp}, by formally taking the limit $N\rightarrow \infty$ in the equation~\eqref{eqY}, we consider the equation
\begin{gather}
\begin{cases}
\partial_t Y_t = \frac{1}{2} (\Delta - 1)Y_t - \frac{1}{2}\int_{[-\alpha_0,\alpha_0]}\alpha e^{\alpha  Y_t}\mathcal{X}_t^\alpha \nu(d\alpha)  \label{eqY2} \\
Y_0 = \eta. 
\end{cases}
\end{gather}
We prove the unique existence of the mild solution to the equation~\eqref{eqY2} in the solution space 
\[ \mathscr{Y}_T \coloneqq \left\{ \mathcal{Y} \in L^{1}([0,T]; H^1(\Lambda)\cap C(\Lambda))\cap C([0,T];H^{-\beta}(\Lambda))\ ;\ e^{\alpha_0|\mathcal{Y}|}\in L^2([0,T];C(\Lambda) )   \right\}   \]
for any $\beta \in (\frac{\alpha_0^2}{4\pi},1)$ and $T>0$.
We use a standard compactness argument. More precisely, we first derive estimates on $Y^N$ (and $e^{\alpha Y^N}$) that are uniform in $N\in \mathbb{N}$ on suitable function spaces. Then, combining them with compact embedding theorems and the uniqueness of the solution to the limit equation, we obtain the desired result, where we also need to use some stochastic estimates. The proofs of the stochastic estimates will be postponed to Section~\ref{sec:stoch}.

We first prove the following uniqueness result, which is suitable for our problem in view of Lemma \ref{lem:weightedexp}.
\begin{lemm}\label{lem:unique}
There is at most one mild solution $\mathcal{Y} \in \mathscr{Y}_T$ of the (deterministic) equation 
\begin{gather}
\begin{cases}
\partial_t \mathcal{Y}_t = \frac{1}{2} (\Delta - 1)\mathcal{Y}_t - \frac{1}{2}\int_{[-\alpha_0,\alpha_0]}\alpha e^{\alpha  \mathcal{Y}_t}\mathcal{Z}_t^\alpha \nu(d\alpha)  \label{eqY3} \\
\mathcal{Y}_0 = \eta 
\end{cases}
\end{gather}
for any given $\eta \in H^{2-\beta}(\Lambda)$ and $\mathcal{Z} \in L^2([0,T]\times [-\alpha_0, \alpha_0], \lambda_T \otimes \nu;H_+^{-\beta}(\Lambda))$.
\end{lemm}
\begin{proof}
Let $\mathcal{Y}$ and $\mathcal{Y}'$ be two solutions of \eqref{eqY3} with the same initial condition $\eta$. Then, $Z \coloneqq \mathcal{Y}-\mathcal{Y}'$ solves the equation 
\begin{equation}\label{eq:differ}
\left\{ \partial_t - \frac{1}{2}(\Delta - 1)  \right\}Z_t = -\frac{1}{2}\int \alpha \left(e^{\alpha \mathcal{Y}_t } - e^{\alpha \mathcal{Y}'_t}  \right) \mathcal{Z}_t^\alpha \nu(d\alpha) \eqqcolon D_t    
\end{equation}
with $Z_0=0$.
Since $e^{\alpha_0 {|\mathcal{Y}|}}, e^{\alpha_0 |\mathcal{Y}'|} \in L^2([0,T]; C(\Lambda))$ and $\mathcal{Z}^\alpha \in L^2([0,T]\times [-\alpha_0,\alpha_0],\lambda_T \otimes \nu;H_+^{-\beta}(\Lambda))$, we have $D \in L^1([0,T]; H^{-\beta}(\Lambda))$ by Lemma \ref{lem:prod}. From the equation \eqref{eq:differ} and $Z_0=0$, we can justify the following calculation by the same way as in \cite[Lemma 3.8]{HKK21a}. Noting that $Z \in \mathscr{Y}_T$,
\begin{align*}
    \int_{\Lambda} Z_t(x) \arctan{Z_t(x)} dx &= \int^t_0 \int_\Lambda \arctan{Z_s(x)} \partial_s   Z_s(x) dxds + \int^t_0 \int_\Lambda Z_s(x)\partial_s \arctan{Z_s(x)}  dxds\\
   &= \frac{1}{2} \int^t_0 \int_\Lambda \left(\arctan{Z_s(x)} + \frac{Z_s(x)}{1 + |Z_s(x)|^2} \right) (\Delta-1)Z_s(x) dxds \\
   &\ \ \ + \int^t_0 \int_\Lambda \left(\arctan{Z_s(x)} + \frac{Z_s(x)}{1 + |Z_s(x)|^2} \right) D_s(x) dxds \\
   &= -\int^t_0 \int_\Lambda \frac{|\nabla Z_s(x)|^2}{1 + |Z_s(x)|^2}dxds - \int^t_0 \int_\Lambda \frac{|Z_s(x)|^2  |\nabla Z_s(x)|^2}{(1 + |Z_s(x)|^2)^2}dxds \\
  &\ \ \  -\frac{1}{2}\int^t_0 \int_\Lambda \left(Z_s(x)\arctan{Z_s(x)} + \frac{|Z_s(x)|^2}{1+ |Z_s(x)|^2}   \right) dxds \\
   &\ \ \ + \int_0^t \int_\Lambda \left( \arctan{Z_s(x)} + \frac{Z_s(x)}{1 + |Z_s(x)|^2}  \right) D_s(x) dxds \\
   &\le \int_0^t \int_\Lambda \left( \arctan{Z_s(x)} + \frac{Z_s(x)}{1 + |Z_s(x)|^2}  \right) D_s(x) dxds.
   \end{align*}
   Moreover, we have
\begin{align*}
    &\int_0^t \int_\Lambda \left( \arctan{Z_s(x)} + \frac{Z_s(x)}{1 + |Z_s(x)|^2}  \right) D_s(x) dxds \\
    &= -\frac{1}{2}\int_\Lambda \left[ \int \alpha \left( e^{\alpha \mathcal{Y}_s}- e^{\alpha \mathcal{Y}'_s}  \right)\left( \arctan{Z_s(x)} + \frac{Z_s(x)}{1 + |Z_s(x)|^2}  \right)\mathcal{Z}_s^\alpha \nu(d\alpha)  \right]    dx \\
    &= -\frac{1}2 \int_\Lambda \left[  \int \alpha^2 e^{A(\alpha \mathcal{Y}_s(x), \alpha \mathcal{Y}'_s(x))}\left( Z_s(x)\arctan{Z_s(x)} + \frac{|Z_s(x)|^2}{1 + |Z_s(x)|^2}  \right)\mathcal{Z}_s^\alpha \nu(d\alpha)\right] dx \\
    &\le 0,
\end{align*}
where $A(x,y)$ is a continuous function on $\mathbb{R}^2$ defined by
\begin{align*}
A(x,y)\coloneqq
    \begin{cases}
        \log \frac{e^x - e^y}{x-y}\ \ \ \mbox{if}\ x\neq y \\
        x\ \ \ \ \ \ \ \ \ \ \ \ \ \mbox{if}\ x=y.
    \end{cases}
\end{align*}
Therefore, we get $Z=0$ and thus $\mathcal{Y}=\mathcal{Y}'$.

\end{proof}

Now, we turn to the proof of the existence of the solution. As already mentioned, we need some estimates on the approximation equation.
\begin{lemm}\label{lem:esty}
    For any $\beta \in (\frac{\alpha_0^2}{4\pi},1)$, $\delta\in(0,1-\beta)$ and $r\in (1,2)$, we have
    \begin{align*}
&\|Y^N \|_{L^r([0,T];H^{1+\delta}(\Lambda) ) \cap C([0,T];H^{1+\delta - \frac{2}r}) \cap C^{\frac{\delta}{2}}([0,T];H^{1 -\frac{2}{r}}(\Lambda) ) }\\
&\lesssim  \|\eta \|_{H^{1+\delta}(\Lambda)} +\| e^{\alpha_0 |Y^N| }\|_{L^{\frac{2r}{2-r}}([0,T];C(\Lambda))} \int_{[-\alpha_0,\alpha_0]}\| \mathcal{X}^{\alpha,N} \|_{L^2([0,T];H^{-\beta}(\Lambda))}\nu(d\alpha),
    \end{align*}
    where the implicit constant does not depend on $N$.
\end{lemm}

\begin{proof}
From the Schauder estimate (Lemma \ref{lem:schauder}) and the product estimate (Lemma \ref{lem:prod}), we have
\begin{align*}
&\|Y^N \|_{L^r([0,T];H^{1+\delta}(\Lambda) ) \cap C([0,T];H^{1+\delta - \frac{2}r}) \cap C^{\frac{\delta}{2}}([0,T];H^{1 -\frac{2}{r}}(\Lambda) ) } \\
&\lesssim \|\eta \|_{H^{1+\delta}(\Lambda)} + \int_{[-\alpha_0,\alpha_0]} |\alpha|\left\|e^{\alpha Y^N}\mathcal{X}^{\alpha,N}\right\|_{L^r([0,T];H^{-\beta}(\Lambda))} \nu(d\alpha)  \\
&\lesssim_{\alpha_0} \|\eta \|_{H^{1+\delta}(\Lambda)} + \int_{[-\alpha_0,\alpha_0]}\| e^{\alpha Y^N }\|_{L^{ \frac{2r}{2-r}   }([0,T];C(\Lambda))}\| \mathcal{X}^{\alpha,N} \|_{L^2([0,T];H^{-\beta}(\Lambda))}\nu(d\alpha)\\
&\le   \|\eta \|_{H^{1+\delta}(\Lambda)} +\| e^{\alpha_0 |Y^N| }\|_{L^{  \frac{2r}{2-r}  }([0,T];C(\Lambda))} \int_{[-\alpha_0,\alpha_0]}\| \mathcal{X}^{\alpha,N} \|_{L^2([0,T];H^{-\beta}(\Lambda))}\nu(d\alpha).
\end{align*}
  
\end{proof}

\begin{lemm}\label{lem:estey}
    For any $p>1$ and $\epsilon>0$, we have
    \begin{align}
&\| e^{\alpha_0 |Y^N| }\|_{L^p([0,T];C(\Lambda))}    \lesssim 1 + \|e^{ \alpha_0 | \eta|} \|_{L^\infty(\Lambda)} + \int_{[-\alpha_0, \alpha_0]\backslash\{0\}}\|C (\nu) \alpha\mathcal{X}^{\alpha, N}  \|_{L^{\frac{\alpha_0 p}{|\alpha|}}([0,T];H^{-1+\epsilon}(\Lambda))}^{\frac{\alpha_0 p}{|\alpha|}}\nu(d\alpha) \label{est:ey}
    \end{align}
    uniformly in $N\in \mathbb{N}$ for some $C(\nu)>0$.
    \end{lemm}
    \begin{rem}\label{rem:estey}
        We prove in Lemma \ref{lem:stoch} that the third term in the right-hand side of \eqref{est:ey} is bounded in $N\in \mathbb{N}$ almost surely if $p\le \sqrt\frac{16\pi}{\alpha_0^2}$, $\alpha_0^2<4\pi$ and $\epsilon>0$ is sufficiently small. 
    \end{rem}
    
\begin{rem}\label{rem:estey2}
\begin{enumerate}
\item
If either $\mbox{supp}(\nu) \subset [-\alpha_0,0]$ or $\mbox{supp}(\nu) \subset [0,\alpha_0]$ holds, we actually have a ``better'' version of Lemma~\ref{lem:estey} by the following argument. 
Writing the equation \eqref{eqY} in the mild form, 
\begin{align}
    Y_t^N = e^{\frac{t}2(\Delta -1)}\eta - \frac{1}{2}\int^t_0 e^{\frac{1}2(t-s)(\Delta -1) }\int_{[-\alpha_0,\alpha_0]}\alpha \exp\left(\alpha (X_t^N + Y_t^N)-\frac{\alpha^2}2C_N  \right)\nu(d\alpha)ds.\notag
\end{align}
Therefore, if we have either $\mbox{supp}(\nu) \subset [-\alpha_0,0]$ or $\mbox{supp}(\nu) \subset [0,\alpha_0]$, it holds for any $\gamma \in \mbox{supp}(\nu)$ that 
\begin{align}
    \gamma Y_t^N = \gamma e^{\frac{t}2(\Delta -1)}\eta - \frac{1}{2}\int^t_0 e^{\frac{1}2(t-s)(\Delta -1) }\int_{\mbox{supp}(\nu)} \gamma\alpha \exp\left(\alpha (X_t^N + Y_t^N)-\frac{\alpha^2}2C_N  \right)\nu(d\alpha)ds \le \gamma e^{\frac{t}2(\Delta -1)}\eta \notag
    \end{align}
by the positivity of the heat kernel. 
In particular, we have
\begin{align}
    \| e^{\alpha Y_t^N} \|_{L^\infty([0,T];C(\Lambda))} \le e^{\alpha_0 \| \eta \|_{L^\infty(\Lambda)}}
\end{align}
for any $\alpha \in \mbox{supp}(\nu)$. This estimate allows us to directly apply the method used in \cite{HKK23} and consequently, we get the convergence of $Y^N$ in the full $L^1$-regime $\alpha_0^2<8\pi$ in this case.

    \item 
    If we have neither $\mbox{supp}(\nu) \subset [-\alpha_0,0]$ nor $\mbox{supp}(\nu) \subset [0,\alpha_0]$, the techniques in (i) is unavailable, but we can still get the estimate \eqref{est:ey}. However, that forces us to work in the $L^2$-regime $\alpha_0^2<4\pi$ and we do not know how to go beyond it at this point.
    \end{enumerate}

\end{rem}

\begin{proof}[Proof of Lemma \ref{lem:estey}]
For any $\gamma \in \mathbb{R}$, we have 
\[ \partial_t e^{\gamma Y_t^N} = (\partial_t Y_t^N)\gamma e^{\gamma Y_t^N}\ \ \ \mbox{and}\ \ \ \Delta e^{\gamma Y_t^N} = (\Delta Y_t^N)\gamma e^{\gamma Y_t^N} + |\nabla Y_t^N|^2 \gamma^2 e^{\gamma Y_t^N}     .\]
Therefore, since $Y^N \in C^\infty(\Lambda)$ solves the equation \eqref{eqY}, it holds that 
\begin{align}
    \left(\partial_t - \frac{1}2 \Delta \right)e^{\gamma Y_t^N}  &= \left\{\partial_t Y_t^N + \frac{1}2(1-\Delta)Y_t^N  - \frac{1}2 Y_t^N - \frac{\gamma}2|\nabla Y_t^N|^2\right\}\gamma e^{\gamma Y_t^N} \notag \\
    &= -\frac{\gamma^2}2  e^{\gamma Y_t^N}|\nabla Y_t^N|^2 -\frac{\gamma}2 Y_t^N e^{\gamma Y_t^N}  - \frac{\gamma}2e^{\gamma Y_t^N } \int_{[-\alpha_0,\alpha_0]} \alpha e^{\alpha Y_t^N} \mathcal{X}^{\alpha,N}_t \nu(d\alpha) \notag     \\
    &\le e^{-1} + \frac{\gamma}2 \int_{I(\gamma)} \alpha e^{(\gamma+\alpha)Y_t^N}\mathcal{X}_t^{\alpha,N}\nu(d\alpha), \notag
\end{align}
where we define
\begin{gather}
I(\gamma) \coloneqq
\begin{cases}
[-\alpha_0,0)\ \ \ \mbox{if}\ \gamma>0; \\
(0,\alpha_0]\ \ \ \ \ \ \mbox{if}\ \gamma<0
\end{cases}
\end{gather}
and used the inequality $xe^x \ge - e^{-1}$.
In particular, letting $\gamma = \theta \alpha_0$, $\theta \in \{1, -1 \}$, we have
\begin{align*}
\left(\partial_t - \frac{1}2 \Delta \right)e^{\theta \alpha_0 Y_t^N}  &\le 1 + \frac{\theta\alpha_0}2 \int_{I(\theta)} \alpha e^{(\theta\alpha_0+\alpha)Y_t^N}\mathcal{X}_t^{\alpha,N}\nu(d\alpha).
\end{align*}
Therefore, from the positivity of the heat kernel, we have
\begin{align}
   &\| e^{\theta \alpha_0 Y_t^N} \|_{C(\Lambda)}\notag \\
   &\le t + \| e^{\theta \alpha_0 \eta} \|_{L^\infty (\Lambda)} + \left\| \int^t_0 e^{\frac{1}2(t-s)\Delta }\left\{ \frac{\theta\alpha_0}2 \int_{I(\theta)} \alpha e^{(\theta\alpha_0+\alpha)Y_s^N}\mathcal{X}_s^{\alpha,N}\nu(d\alpha)\right\} ds\right\|_{C(\Lambda)}.\label{est:ey1}
\end{align}
Moreover, from the Sobolev embedding $H^{1+\frac{\epsilon}{2}}(\Lambda) \subset C(\Lambda)$ and the Schauder estimate (Lemma \ref{lem:schauder}) and shorthand notations such as $L_T^pH^s(\Lambda)\coloneqq L^p([0,T];H^s(\Lambda))$, it holds that 
\begin{align}
 &\left\| \int^t_0 e^{\frac{1}2(t-s)\Delta }\left\{ \frac{\theta\alpha_0}2 \int_{I(\theta)} \alpha e^{(\theta\alpha_0+\alpha)Y_s^N}\mathcal{X}_s^{\alpha,N}\nu(d\alpha)\right\} ds\right\|_{L_T^pC(\Lambda)}\notag\\
 &\lesssim  \left\| \int^t_0 e^{\frac{1}2(t-s)\Delta }\left\{ \frac{\theta\alpha_0}2 \int_{I(\theta)} \alpha e^{(\theta\alpha_0+\alpha)Y_s^N}\mathcal{X}_s^{\alpha,N}\nu(d\alpha)\right\} ds\right\|_{L_T^pH^{1+ \frac{\epsilon}{2}}(\Lambda)}\notag\\
 &\lesssim_{\alpha_0}  \left\|\int_{I(\theta )} \alpha e^{(\theta\alpha_0+\alpha)Y_t^N}\mathcal{X}_t^{\alpha,N}\nu(d\alpha)\right\|_{L_T^p H^{-1+  \epsilon}(\Lambda)} \notag    \\
 &\lesssim_{\alpha_0} \int_{I(\theta)} \left\|\alpha e^{(\theta\alpha_0+\alpha)Y_t^N}\mathcal{X}_t^{\alpha,N}\right\|_{L_T^p H^{-1+  \epsilon}(\Lambda)}\nu(d\alpha).\label{est:ey2}
 \end{align}
 From the product estimate (Lemma \ref{lem:prod}), for $\alpha \in I(\theta )$
 \begin{align}
     \left\|\alpha e^{(\theta\alpha_0+\alpha)Y_t^N}\mathcal{X}_t^{\alpha,N}\right\|_{L_T^p H^{-1+ \delta + \epsilon}(\Lambda)} &\le \|e^{(\theta\alpha_0 + \alpha)Y_t^N} \|_{L_T^{\frac{\theta\alpha_0}{\theta\alpha_0 + \alpha}p }C(\Lambda)} \| \alpha \mathcal{X}_t^{\alpha,N}\|_{L_T^{p(\alpha)}H^{-1+ \epsilon}(\Lambda) }\notag \\
     &= \| e^{\theta \alpha_0 Y_t^N}\|_{L_T^{p}C(\Lambda)}^{\frac{\theta\alpha_0 + \alpha}{\theta \alpha_0}} \|\alpha \mathcal{X}_t^{\alpha,N}\|_{L_T^{p(\alpha)}H^{-1+ \epsilon}(\Lambda) },\label{est:ey3}
     \end{align}
 where $p(\alpha)$ is defined by the equality
 \[ \frac{\theta \alpha_0 + \alpha}{\theta \alpha_0 p} + \frac{1}{p(\alpha)} = \frac{1}{p} \iff p(\alpha)= -  \frac{\theta\alpha_0p}{\alpha} = \frac{\alpha_0p}{|\alpha|}.   \]
 Therefore, we have from \eqref{est:ey1}--\eqref{est:ey3} and Young's inequality that 
 \begin{align*}
   &\| e^{\theta \alpha_0 Y_t^N} \|_{L_T^p C(\Lambda)} \\
   &\le t+ \| e^{\theta \alpha_0\eta} \|_{L^\infty (\Lambda)} + C\int_{I(\theta)}\| e^{\theta\alpha_0 Y_t^N}\|_{L_T^{p}C(\Lambda)}^{\frac{\theta\alpha_0 + \alpha}{\theta \alpha_0}} \|\alpha \mathcal{X}_t^{\alpha,N}\|_{L_T^{p(\alpha)}H^{-1+ \epsilon}(\Lambda) } \nu(d\alpha)\\
   &\le t+ \| e^{\theta \alpha_0\eta} \|_{L^\infty (\Lambda)}\\
   &\ + C \int_{I(\theta)} \left\{ \frac{\theta \alpha_0 + \alpha}{\theta \alpha_0} \lambda(\alpha) \| e^{\theta \alpha_0 Y^N}\|_{L_T^p C(\Lambda)} - \frac{\theta\alpha}{\alpha_0} \left( \lambda(\alpha)^{-\frac{\theta \alpha_0 + \alpha}{\theta \alpha_0}} \|  \alpha \mathcal{X}_t^{\alpha, N}  \|_{L_T^{p(\alpha)}H^{-1+\epsilon}(\Lambda)}  \right)^{\frac{\theta \alpha_0}{-\alpha}}\right\} \nu(d\alpha)
   \end{align*}
   for some $C>0$ and any $\lambda(\alpha)>0$.
   Thus, letting $\lambda(\alpha) \coloneqq \frac{1}{2C(\nu[-\alpha_0,\alpha_0] + 1)}\times \frac{\theta \alpha_0}{\theta \alpha_0 + \alpha}$, we get
   \begin{align*}
       &\left( 1 -  \frac{\nu[-\alpha_0,\alpha_0]}{2(\nu[-\alpha_0,\alpha_0]+1)}   \right)\| e^{\theta \alpha_0 Y_t^N} \|_{L_T^p C(\Lambda)} \\
       &\le t+  \| e^{\theta \alpha_0\eta} \|_{L^\infty (\Lambda)} + C\int_{I(\theta)} \frac{|\alpha|}{\alpha_0} \left\{  2C(\nu[-\alpha_0,\alpha_0] + 1)\frac{\theta \alpha_0 + \alpha}{\theta \alpha_0}   \right\}^{\frac{\theta \alpha_0 + \alpha}{\theta \alpha_0}\times \frac{\theta \alpha_0}{-\alpha}}\| \alpha \mathcal{X}_t^{\alpha, N} \|_{L_T^{p(\alpha)}H^{-1+\epsilon}(\Lambda)}^{\frac{\theta \alpha_0}{-\alpha}}\nu(d\alpha)\\
       &\le T+ C_1(\nu) +  \| e^{\theta \alpha_0\eta} \|_{L^\infty (\Lambda)} + C\int_{I(\theta)} \| C_2(\nu)\alpha \mathcal{X}_t^{\alpha, N} \|_{L_T^{p(\alpha)}H^{-1+\epsilon}(\Lambda)}^{p(\alpha)}\nu(d\alpha)
       \end{align*}
       for some constant $C_1(\nu)$ and 
       \[ C_2(\nu) \coloneqq \sup_{\alpha\in I(\theta)}\left\{ 2C(\nu[-\alpha_0,\alpha_0]+1)\frac{\theta\alpha_0 + \alpha}{\theta \alpha_0} \right\}^{\frac{\theta\alpha_0 +\alpha}{\theta\alpha_0}}<\infty.    \]
\end{proof}

Finally, we prove the convergence of $Y^N$ and the existence of the solution to \eqref{eqY2}.
\begin{lemm}\label{lem:conv}
    Let $\alpha_0^2<4\pi$. For any $\beta \in (\frac{\alpha_0^2}{4\pi},1)$, $\eta\in H^{2-\beta}(\Lambda)$ and $T>0$, $Y^N$ converges to the unique mild solution $Y \in \mathscr{Y}_T$ of  \eqref{eqY2}  in $L^1([0,T];H^{1}(\Lambda)\cap C(\Lambda) )\cap C([0,T];H^{-\beta}(\Lambda))$ almost surely as $N\rightarrow \infty$. 
    \end{lemm}
    \begin{rem}\label{rem:lemconv}
        From Lemmas \ref{lem:conv} and \ref{lem:ou-wick}, we get that $\Phi^N = X^N + Y^N$ converges to $\Phi \coloneqq X+Y$ in $C([0,T];H^{-\beta}(\Lambda))$ almost surely as $N\rightarrow \infty$, which is the statement of Theorem \ref{thm1}.
    \end{rem}
\begin{proof}[Proof of Lemma \ref{lem:conv}]
We proved in Lemma \ref{lem:esty} that for any $r\in (1,2)$ and $\delta \in (0,1-\beta)$, we have
\begin{align}
&\|Y^N \|_{L^r([0,T];H^{1+\delta}(\Lambda) ) \cap C([0,T];H^{1+\delta - \frac{2}r}) \cap C^{\frac{\delta}{2}}([0,T];H^{1 -\frac{2}{r}}(\Lambda) ) }\notag\\
&\lesssim  \|\eta \|_{H^{1+\delta}(\Lambda)} +\| e^{\alpha_0 |Y^N| }\|_{L^{\frac{2r}{2-r}}([0,T];C(\Lambda))} \int_{[-\alpha_0,\alpha_0]}\| \mathcal{X}^{\alpha,N} \|_{L^2([0,T];H^{-\beta}(\Lambda))}\nu(d\alpha).\label{est:yn}
\end{align}
Moreover, by taking $r$ close enough to 1, we have from Lemmas \ref{lem:estey} and \ref{lem:stoch} that the right-hand side of \eqref{est:yn} is bounded in $N\in\mathbb{N}$ almost surely. Therefore, letting $(\Omega,\mathcal{F},\mathbb{P})$ be the underlying probability space on which the driving noise $W$ of the equation is defined, there exists an event $\Omega_0 \in \mathcal{F}$ with $\mathbb{P}(\Omega_0)=1$ such that for any $\omega\in \Omega_0$ we have
\begin{align*}
    \sup_{N\in \mathbb{N}}\|Y^N (w)\|_{L^r([0,T];H^{1+\delta}(\Lambda) ) \cap C([0,T];H^{1+\delta - \frac{2}r}) \cap C^{\frac{\delta}{2}}([0,T];H^{1 -\frac{2}{r}}(\Lambda) ) } < \infty.
    \end{align*}
In the following, we fix $\omega \in \Omega_0$.
Since the embeddings 
\[  L^1([0,T];H^{1+\delta}(\Lambda))\cap C^{\frac{\delta}{2}}([0,T];H^{-1}(\Lambda)) \hookrightarrow L^1([0,T]; H^{1+\delta'}(\Lambda)) ,    \]
\[C([0,T];H^{\delta-1}(\Lambda))\cap C^{\frac{\delta}{2}}([0,T];H^{-1}(\Lambda)) \hookrightarrow C([0,T]; H^{-\beta'}(\Lambda))\]
are compact for any $-\beta>\delta-1>-\beta'>-1$ and $\delta>\delta'$, (see \cite[Lemma 3.9]{HKK21a}) there exists a subsequence $Y^{N_k}$ such that it converges to some $Y$ in $L^1([0,T];H^{1+\delta'}(\Lambda) )\cap C([0,T];H^{-\beta'}(\Lambda))$. 
Once we prove that $Y$ solves the equation \eqref{eqY} in the space $\mathscr{Y}_T$, we get the desired result from the uniqueness of the solution to the equation, which we proved in Lemma \ref{lem:unique}. In the following, we prove that $e^{\alpha Y^{N_k}}$ converges to $e^{\alpha Y}$ in $L^2_{T,\alpha_0} C(\Lambda) \coloneqq L^2(\lambda_T\otimes \nu;C(\Lambda))$, where $\lambda_T$ is the Lebesgue measure on $[0,T]$.
Then, in view of Lemmas \ref{lem:weightedexp} and \ref{lem:prod}, by taking the limit in the mild form of the approximation equation 
\begin{align*}
  Y^{N_k}_t = e^{\frac{t}2 (\Delta-1)} \eta + \int^t_0 e^{\frac{t-s}{2}(\Delta-1)} \int_{[-\alpha_0,\alpha_0]}\alpha e^{\alpha Y^{N_k}_s} \mathcal{X}^{\alpha,N_k}_s\nu(d\alpha)ds,  
\end{align*}
we can see that $Y$ is the mild solution of the equation \eqref{eqY2}. 
Note that we now have
\begin{equation}\label{eq:unif}
\sup_{k\in \mathbb{N}} \|e^{\alpha Y^{N_k}}   \|_{L^{2+\epsilon}_{T,\alpha_0}C(\Lambda)} \lesssim_\nu   \sup_{k\in \mathbb{N}}\| e^{\alpha_0 |Y^{N_k}|} \|_{L_T^{\frac{2r}{2-r}}C(\Lambda)} < \infty 
\end{equation}
for sufficiently small $\epsilon>0$. On the other hand, from the convergence $Y^{N_k}\rightarrow Y$ in $L_T^1 H^{1+\delta'}(\Lambda)$, we have
\begin{align*}
\left\|\alpha Y^{N_k} - \alpha Y   \right\|_{L_{T,\alpha_0}^1 H^{1+\delta'}(\Lambda)} = \int |\alpha| \nu(d\alpha)\left\| Y^{N_k} -  Y   \right\|_{L_{T}^1 H^{1+\delta'}(\Lambda)} \rightarrow 0
\end{align*}
as $k\rightarrow \infty$.
Therefore, by taking a subsequence, we can assume that $\alpha Y^{N_k}_t$ converges to $\alpha Y_t$ in $H^{1+\delta'}(\Lambda) \subset C(\Lambda)$ for $\lambda_T \otimes \nu$-almost every $(t,\alpha)\in [0,T]\times [-\alpha_0,\alpha_0]$ and thus $e^{\alpha Y^{N_k}_t}$ converges to $e^{\alpha Y_t}$ in $C(\Lambda)$ for $\lambda_T \otimes \nu$-almost everywhere.  Therefore, from the convergence theorem for uniformly integrable functions, we get the desired convergence. Note that $Y\in\mathscr{Y}_T$ follows from
\[  \| e^{\alpha_0 |Y|}\|_{L^2_TC(\Lambda)} \le \liminf_{k\rightarrow \infty}\| e^{\alpha_0 |Y^{N_k}|}\|_{L^{2}_TC(\Lambda)}<\infty  \]
by Fatou's lemma and \eqref{eq:unif}.
\end{proof}

\section{Stochastic estimates}\label{sec:stoch}
In this section, we prove Lemmas \ref{lem:weightedexp} and \ref{lem:stoch} below, which we used in Section \ref{sec:gwp}. The proof of Lemma \ref{lem:weightedexp} is similar to the proof of Lemma \ref{lem:stoch} and simpler, so we omit it.
Recall that we defined
\[ \mathcal{X}^{\alpha,N}_t \coloneqq \exp_N^{\Diamond}{(\alpha X^N_t)} \coloneqq \exp\left( \alpha P_N X_t - \frac{\alpha^2}2 C_N \right)   \]
and $P_N$ was introduced in \eqref{eq:appro}.

\begin{lemm}\label{lem:stoch}
   Let $\alpha_0^2<4\pi$. We have
    \begin{align*}
        \sup_{N\in \mathbb{N}} \int_{[-\alpha_0, \alpha_0]\backslash\{0\}} \left\|C(\alpha)  \mathcal{X}^{\alpha,N}\right\|^{p(\alpha)}_{L^{p(\alpha)}([0,T];H^{-1+\epsilon}(\Lambda))} \nu(d\alpha) < \infty
    \end{align*}
    almost surely for sufficiently small $\epsilon>0$, where $p(\alpha) \coloneqq \frac{p\alpha_0 }{|\alpha|}$, $p\in[2,\sqrt {\frac{16\pi}{\alpha_0^2}})$ and $C(\alpha)$ is any real-valued function such that $0\le C(\alpha) \lesssim |\alpha|$.
\end{lemm}
\begin{rem}
    Note that if the finite Borel measure $\nu$ is of the form
    \[   \nu = \sum_{m=1}^M a_m \delta_{\alpha_m},\ \ \ a_m\in\mathbb{R}_+, \ \alpha_m \in [-\alpha_0,\alpha_0], \ M\in\mathbb{N}, \]
    the assertion of Lemma \ref{lem:stoch} directly follows from Lemma \ref{lem:ou-wick} (iii).
\end{rem}
Before proceeding to the proof of Lemma \ref{lem:stoch}, we prepare the following simple lemma.
\begin{lemm}\label{lem:sum}
    Let $p\ge 1$ and $0<c<1$. One has
    \[ \sum_{n=1}^\infty n^p c^{pn} \le d(c)^p  \]
    for some $d(c)>0$ which depends only on $c$.
\end{lemm}
\begin{proof}
Noting that $n^{\frac{1}n} \searrow 1$ as $n\rightarrow \infty$, we fix any $N\in\mathbb{N}$ such that 
\[ N^{\frac{1}{N}} \le \frac{c+1}{2c}.  \]
Then, we have
\begin{align*}
\sum_{n=1}^\infty n^p c^{pn} &= \sum_{n=1}^N n^p c^{pn} + \sum_{n=N+1}^\infty n^p c^{pn} \le \sum_{n=1}^N N^p 1^{pn} + \sum_{n=N+1}^\infty \left( n^{\frac{1}n} c \right)^{pn}\\
&\le N\cdot N^p + \sum_{n=N+1}^\infty \left( \frac{c+1}{2c}c  \right)^{pn} = N^{p+1} + \left(\frac{c+1}{2}\right)^{(N+1)p} \frac{1}{1 - \left(\frac{c+1}{2}\right)^p } \le d(c)^p
\end{align*}
for some $d(c)>0$.    
\end{proof}

\begin{proof}[Proof of Lemma \ref{lem:stoch}]
It is enough to prove the following two estimates:
\begin{align}
   \mathbb{E}[S_N] \coloneqq \mathbb{E}\left[ \int_{[-\alpha_0,\alpha_0]\backslash\{0\}} \int^T_0 2^{p(\alpha)N}C(\alpha)^{p(\alpha)} \left\| \mathcal{X}^{\alpha,N+1}_t - \mathcal{X}_t^{\alpha,N}  \right\|^{p(\alpha)}_{H^{-1+\epsilon}(\Lambda)}dt \nu(d\alpha)  \right] \lesssim 4^{-N} \label{est:stoch}
\end{align}
and 
\begin{align}
    \mathbb{E}\left[\int_{[-\alpha_0,\alpha_0]\backslash\{0\}} \int^T_0 C(\alpha)^{p(\alpha)} \left\|  \mathcal{X}_t^{\alpha,1}  \right\|^{p(\alpha)}_{H^{-1+\epsilon}(\Lambda)}dt \nu(d\alpha)\right]<\infty. \label{est:stoch2}
\end{align}
Indeed, from \eqref{est:stoch} and Chebyshev's inequality, we have for any $c>0$ that
\begin{align*}
 \mathbb{P}\left( 2^N S_N \ge c\right) &\le \frac{1}{c}\mathbb{E} \left[2^N S_N\right] \lesssim 2^N 4^{-N} = 2^{-N}.
\end{align*}
Therefore, from the Borel-Cantelli lemma, it holds that $ 2^N S_N \rightarrow 0$ as $N\rightarrow \infty$ almost surely. In particular, we have
\begin{align}
    S_N = \int_{[-\alpha_0,\alpha_0]\backslash\{0\}} \int^T_0 2^{p(\alpha)N}C(\alpha)^{p(\alpha)} \left\| \mathcal{X}^{\alpha,N+1}_t - \mathcal{X}_t^{\alpha,N}  \right\|^{p(\alpha)}_{H^{-1+\epsilon}(\Lambda)}dt \nu(d\alpha) \lesssim 2^{-N}
\end{align}
almost surely. Thus, combining it with the elementary inequality $(|a| + |b|)^p \le 2^{p-1}(|a|^p + |b|^p)$, we have
\begin{align*}
    &\int_{[-\alpha_0, \alpha_0]\backslash\{0\}} \left\|C(\alpha)  \mathcal{X}^{\alpha,N}\right\|^{p(\alpha)}_{L^{p(\alpha)}([0,T];H^{-1+\epsilon}(\Lambda))} \nu(d\alpha)\\
    &= \int_{[-\alpha_0, \alpha_0]\backslash\{0\}} \left\|C(\alpha)  \mathcal{X}^{\alpha,1}  + C(\alpha)\sum_{M=1}^{N-1} \left( \mathcal{X}^{\alpha,M+1} - \mathcal{X}^{\alpha,M}  \right) \right\|^{p(\alpha)}_{L^{p(\alpha)}([0,T];H^{-1+\epsilon}(\Lambda))} \nu(d\alpha)\\
    &\le \int_{[-\alpha_0,\alpha_0]\backslash\{0\}}\nu(d\alpha)\left[ 2^{p(\alpha)}\|C(\alpha) \mathcal{X}^{\alpha,1} \|^{p(\alpha)} + \sum_{M=1}^{N-1} 2^{(M+1)p(\alpha)} \|C(\alpha)\left( \mathcal{X}^{\alpha,M+1} - \mathcal{X}^{\alpha,M}  \right)\|^{p(\alpha)}    \right]\\
    &\lesssim \int_{[-\alpha_0,\alpha_0]\backslash\{0\}} \| 2C(\alpha) \mathcal{X}^{\alpha,1}\|^{p(\alpha)}  \nu(d\alpha) + \sum_{M=1}^{N-1}2^{-M}
    \end{align*}
    almost surely. Then, noting that the first term in the last line is finite almost surely in view of \eqref{est:stoch2}, the assertion of the lemma follows.

    The remaining part of the proof is devoted to the proof of \eqref{est:stoch}. The proof of \eqref{est:stoch2} is similar and simpler, so we omit it.
    Because $X_t$ is distributed according to $\mu_0$ for any $t\ge 0$ and $\mathcal{X}_t^{\alpha,N} \coloneqq \exp_{N}^{\Diamond}(\alpha X_t)$, it holds that
    \begin{align}
        \mathbb{E}[S_N] &= \mathbb{E} \left[\int_{[-\alpha_0,\alpha_0]\backslash\{0\}} \int^T_0 2^{p(\alpha)N}C(\alpha)^{p(\alpha)} \left\| \mathcal{X}^{\alpha,N+1}_t - \mathcal{X}_t^{\alpha,N}  \right\|^{p(\alpha)}_{H^{-1+\epsilon}(\Lambda)}dt \nu(d\alpha) \right] \notag\\
        &= T \int \nu(d\alpha) 2^{p(\alpha)N}C(\alpha)^{p(\alpha)} \int \mu_0(d\phi) \left\| \exp_{N+1}^{\Diamond}(\alpha \phi) -\exp_{N}^{\Diamond}(\alpha \phi)\right\|^{p(\alpha)}_{H^{-1+\epsilon}(\Lambda)}\notag\\
        &\eqqcolon T \int \nu(d\alpha) 2^{p(\alpha)N}C(\alpha)^{p(\alpha)} E_N(\alpha).\label{est:sn}
        \end{align}
        In the following, we use the notation
        \[   \Wick{(P_N \phi)^n} \coloneqq H_n(P_N \phi,C_N),  \]
        where $H_n(x,c)$ is the $n$-th Hermite polynomial defined in \eqref{herm} and $C_N$ is the renormalization constant in \eqref{eq:const}.
        Then, from H\"older's inequality, 
        \begin{align}
            E_N(\alpha) &= \int  \left\| \exp_{N+1}^{\Diamond}(\alpha \phi) -\exp_{N}^{\Diamond}(\alpha \phi)\right\|^{p(\alpha)}_{H^{-1+\epsilon}(\Lambda)}\mu_0(d\phi) \notag\\
            &= \int  \left[ \int_\Lambda \left| \langle \nabla \rangle^{-1+\epsilon} \sum_{n=1}^\infty \frac{\alpha^n}{n!} \left\{ \Wick{(P_{N+1}\phi)^n} - \Wick{(P_N\phi )^n}  \right\}    \right|^2dx      \right]^{\frac{p(\alpha)}{2}} \mu_0(d\phi)\notag\\
            &\le \int  |\Lambda|^{\frac{p(\alpha)}2 - 1}   \int_\Lambda \left| \langle \nabla \rangle^{-1+\epsilon} \sum_{n=1}^\infty \frac{\alpha^n}{n!} \left\{ \Wick{(P_{N+1}\phi)^n} - \Wick{(P_N\phi )^n}  \right\}    \right|^{p(\alpha)}dx     \mu_0(d\phi)\notag\\
            &\le |\Lambda|^{\frac{p(\alpha)}2 -1} \int  \mu_0(d\phi) \notag\\
            &\ \ \ \int_\Lambda \left(\sum_{n=1}^\infty n^{-\frac{2p(\alpha)}{p(\alpha)-1}}\right)^{p(\alpha)-1}   \left(\sum_{n=1}^\infty \left(\frac{|\alpha|^n}{n!}\right)^{p(\alpha)}n^{2p(\alpha)} \left|\langle \nabla \rangle^{-1+\epsilon} \left\{ \Wick{(P_{N+1}\phi)^n} - \Wick{(P_N\phi )^n}  \right\}\right|^{p(\alpha)}\right)  dx    \notag\\
            &= |\Lambda|^{\frac{p(\alpha)}2 -1} \zeta\left( \frac{2p(\alpha)}{p(\alpha)-1}  \right)^{p(\alpha)-1} \notag\\
            &\ \times\sum_{n=1}^\infty \left(\frac{|\alpha|^n}{n!}\right)^{p(\alpha)}n^{2p(\alpha)}\int_\Lambda \left( \int\left|\langle \nabla \rangle^{-1+\epsilon} \left\{ \Wick{(P_{N+1}\phi)^n} - \Wick{(P_N\phi )^n} \right\}\right|^{p(\alpha)}\mu_0(d\phi)    \right)       dx,
            \end{align}
            where $\zeta(\cdot)$ is the Riemann zeta function.
            Therefore, from the hypercontractivity of the $n$-th homogeneous Wiener chaos (see e.g. \cite[Theorem 5.10]{Jan97}), we get
            \begin{align}
                E_N(\alpha) &\le C_1^{p(\alpha)}\sum_{n=1}^\infty \left(\frac{|\alpha|^n}{n!}\right)^{p(\alpha)}n^{2p(\alpha)} \left( p(\alpha)-1  \right)^{\frac{np(\alpha)}{2}} \notag\\
            &\ \times\int_\Lambda \left( \int\left|\langle \nabla \rangle^{-1+\epsilon} \left\{ \Wick{(P_{N+1}\phi)^n} - \Wick{(P_N\phi )^n}  \right\}\right|^2\mu_0(d\phi)    \right)^{\frac{p(\alpha)}{2} }dx \label{est:en}
            \end{align}
for some constant $C_1>0$ independent of $N, \alpha$. Moreover, according to the proof of \cite[Theorem 2.2]{HKK21a}, it holds that
\begin{align}
\int\left|\langle \nabla \rangle^{-1+\epsilon} \left\{ \Wick{(P_{N+1}\phi)^n} - \Wick{(P_N\phi )^n}  \right\}\right|^2\mu_0(d\phi) \le C_2(\epsilon,\lambda) A^{-\lambda N} (n!)^2\left(\frac{1+\lambda}{4\pi(1-\epsilon -\lambda)}  \right)^n\label{est:HKK}
\end{align}
for some $C_2>0$ and any $0<\lambda<1-\epsilon$, where $A$ is a positive constant introduced in \eqref{eq:appro}. Note that our definition of the approximation operator $P_N$ in \eqref{eq:appro} is the reason why the factor $A^{-\lambda N}$ appears in \eqref{est:HKK} instead of $2^{-\lambda N}$ like in the proof of \cite[Theorem 2.2]{HKK21a}. 
From \eqref{est:en} and \eqref{est:HKK}, we have
\begin{align}
    E_N(\alpha) &\le C_3(\epsilon,\lambda)^{p(\alpha)}A^{-\frac{\lambda p(\alpha)}{2}N}\sum_{n=1}^\infty n^{2p(\alpha)} \left\{ \alpha^2 (p(\alpha)-1) \frac{1+\lambda}{4\pi(1-\epsilon -\lambda)}    \right\}^{\frac{p(\alpha)}{2}n}\notag\\
    &\le C_3(\epsilon,\lambda)^{p(\alpha)}A^{-\frac{\lambda p(\alpha)}{2}N}\sum_{n=1}^\infty n^{2p(\alpha)} \left\{ \frac{p^2\alpha_0^2}4 \frac{1+\lambda}{4\pi(1-\epsilon -\lambda)}    \right\}^{\frac{p(\alpha)}{2}n}\notag\\
    &\le C_3(\epsilon,\lambda)^{p(\alpha)}A^{-\frac{\lambda p(\alpha)}{2}N}\sum_{n=1}^\infty n^{2p(\alpha)}  C_4(\alpha_0)^{\frac{p(\alpha)}{2}n}\label{est:en2}
    \end{align}
    for some $C_3>0$ and $0<C_4<1$ by taking sufficiently small $\epsilon,\lambda>0$, where we used the inequality 
    \[\alpha^2 (p(\alpha)-1) = \alpha^2\left( \frac{p\alpha_0}{|\alpha|}-1  \right)= p\alpha_0|\alpha|- \alpha^2 \le \frac{p^2\alpha_0^2}4 < 4\pi .\]
Therefore, noting that 
\[\sum_{n=1}^\infty n^{2p(\alpha)}  C_4(\alpha_0)^{\frac{p(\alpha)}{2}n} \le C_5(\alpha_0)^{p(\alpha)}\]
for some $C_5>0$ by Lemma \ref{lem:sum},  we get from \eqref{est:sn} and \eqref{est:en2} that 
\begin{align*}
    \mathbb{E} [S_N] \le T\int_{[-\alpha_0,\alpha_0]\backslash\{0\}} 2^{p(\alpha)N}C(\alpha)^{p(\alpha)} C_6(\epsilon,\lambda,\alpha_0)^{p(\alpha)}A^{-\frac{\lambda p(\alpha)}2 N}    \nu(d\alpha)
\end{align*}
for some $C_6>0$. Then, because 
\[ \sup_{\alpha\in [-\alpha_0,\alpha_0]\backslash \{0\}}C(\alpha)^{p(\alpha)} C_6(\epsilon,\lambda,\alpha_0)^{p(\alpha)} =\sup_{\alpha\in [-\alpha_0,\alpha_0]\backslash \{0\}}C(\alpha)^{\frac{p\alpha_0}{|\alpha|}} C_6(\epsilon,\lambda,\alpha_0)^{\frac{p\alpha_0}{|\alpha|}}< \infty\]
from the condition $0\le C(\alpha) \lesssim |\alpha|$, by taking sufficiently large $A$ such that  
\[ 2^{p(\alpha)N}A^{-\frac{\lambda p(\alpha)}2 N} \le 2^{-p(\alpha)N} (\le 4^{-N}) \iff A \ge 2^{\frac{4}{\lambda}},  \]
we get \eqref{est:stoch}.

\end{proof}

\section{A remark on the Dirichlet form approach}\label{sec:dirichlet}
In this section, the proof of Thoerem \ref{thm2} is explained. In the following, we often omit the details and just refer to the relevant parts in \cite{HKK21a,HKK23}, as the many parts of the proof are almost identical to the corresponding proofs in those works.
We begin with the following lemma whose proof is almost the same as \cite[Theorem 2.1]{HKK23}.
Recall that we defined 
\[\exp_N^{\Diamond}(\alpha \phi)\coloneqq \exp\left(\alpha P_N \phi - \frac{\alpha^2}2C_N\right) \]
in Section \ref{sec:pre}.

\begin{lemm}
Let $\alpha_0^2<8\pi$ and let $\nu$ be any nonnegative finite Borel measure on $[-\alpha_0,\alpha_0]$. Then, we have 
\[ \int_{[-\alpha_0,\alpha_0]} \exp_N^{\Diamond}(\alpha \phi)\nu(d\alpha) \rightarrow \int_{[-\alpha_0,\alpha_0]} \exp^{\Diamond}(\alpha \phi)\nu(d\alpha)\]
as $N\rightarrow \infty$ for $\mu_0$-almost every $\phi$ in the Besov space $B_{p,p}^{-\beta}(\Lambda)$ for any $p>1$ and $\beta \in (\frac{\alpha^2}{4\pi}(p-1),1)$.
\end{lemm}
Since $\int \exp^{\Diamond}(\alpha \phi)\nu(d\alpha)$ is a nonnegative distribution, this lemma implies that the probability measure $\mu^{(\nu)}$ in \eqref{eq:measure} is well-defined and absolutely continuous with respect to $\mu_0$.
Therefore, we can consider the pre-Dirichlet form $\mathcal{E}$ defined by
    \begin{equation}
        \mathcal{E}\left(F, G\right) \coloneqq \frac{1}{2}\int_{E} \langle \nabla F (\phi) , \nabla G(\phi)\rangle_{L^2(\Lambda)} \mu^{(\nu)}(d\phi),\ \ F,G \in \mathfrak{F}C^\infty_b \subset L^2(\mu^{(\nu)}),
    \end{equation}
    where $\mathfrak{F}C^\infty_b$ is the space of all bounded smooth cylindrical functionals on $E \coloneqq H^{-\beta}(\Lambda)$ ($\beta\in(0,1)$) defined by
    \begin{align*}
        \mathfrak{F}C_b^\infty \coloneqq \left\{F:E\rightarrow \mathbb{R}\ ;\ \exists m\  \exists f \in C_b^\infty(\mathbb{R}^m)\ \exists \{ l_i \}_{i=1}^m \subset C^\infty(\Lambda)\ \mbox{s.t.}\ F(\phi) = f(\langle \phi, l_1 \rangle, \cdots, \langle \phi, l_m\rangle )   \right\}.
    \end{align*}
    Here, $\langle \cdot,\cdot\rangle$ denotes the duality between $\mathcal{D}'(\Lambda)$ and $\mathcal{D}(\Lambda)$
. The $L^2(\Lambda)$-derivative $\nabla F: E \rightarrow L^2(\Lambda)$ for $F \in \mathfrak{F}C_b^\infty$ of the form in the above definition is given by
    \[ \nabla F(\phi) = \sum_{i=1}^m (\partial_{i}f) (\langle \phi, l_1 \rangle, \cdots, \langle \phi, l_m\rangle)l_i, \ \ \ \phi\in E.     \]
    Then, similarly to the standard $\exp{(\Phi)_2}$-measure, $\mu^{(\nu)}$ satisfies the following integration by parts formula:
    \begin{align}
        \mathcal{E}(F,G) = -\frac{1}{2}\int_E G(\phi) \mathcal{L}^{(\nu)}F(\phi) \mu^{(\nu)}(d\phi),
    \end{align}
    where $\mathcal{L}^{(\nu)}$ is a pre-Dirichlet operator defined by
    \begin{align*}
     \mathcal{L}^{(\nu)}F(\phi) \coloneqq \frac{1}{2}\mbox{Tr}\left( \nabla^2F(\phi) \right) -\frac{1}{2}\left\langle  (1-\Delta)\phi, \nabla F(\phi)\right\rangle-\frac{1}{2}\big\langle \int \alpha \exp^{\Diamond}(\alpha \phi)\nu(d\alpha) , \nabla F(\phi)\big\rangle,
    \end{align*}
see the proof of \cite[Proposition 6.1]{HKK23}.
Moreover, from the same proof as in \cite[Proposition 6.1]{HKK23}, we have $\mathcal{L}^{(\nu)}F(\phi) \in L^2(E,\mu^{(\nu)})$ for any $\alpha_0^2< 8\pi$. Thus, by applying the general method of Dirichlet forms for SPDE developed in \cite{AR91} as in \cite[Section 6]{HKK23}, there exists a suitable $\beta\in(0,1)$ such that $(\mathcal{E},\mathfrak{F}C_b^\infty)$ is closable in $L^2(\mu^{(\nu)})$ and the closure $(\bar{\mathcal{E}},\mathcal{D}(\bar{\mathcal{E}}))$ is a quasi-regular Dirichlet form. In particular, there exists a $E$-valued diffusion process $M=(\Theta, \mathcal{G}, (\mathcal{G})_{t\ge0}, (\Psi_t)_{t\ge 0}, (\mathbb{Q}_\phi)_{\phi\in E})$ properly associated with $\bar{\mathcal{E}}$ in the sense that 
\begin{align}
    P_t F(\phi) = \mathbb{E}^{\mathbb{Q}_\phi}[F(\Psi_t) ]\ \ \ \ \ \mu^{(\nu)}\mbox{-a.e.}\ \phi\in E \label{eqdirichlet1}
\end{align}
for any bounded measurable function $F$ on $E$ and $t\ge 0$, where $(P_t)_{t\ge 0}$ is the Markov semigroup on $L^2(\mu^{(\nu)})$ corresponding to $\bar{\mathcal{E}}$.

From now on, we assume that $\alpha_0^2<4\pi$ and $\beta \in (\frac{\alpha_0^2}{4\pi},1)$.
By a straightforward modification of the argument in \cite[Section 5]{HKK21a} using Lemma \ref{lem:conv}, we can see that there exists a $L^2(\Lambda)$-cylindrical $(\mathcal{G}_t)_{t\ge 0}$-Brownian motion $\mathcal{W}$ on the probability space $(\Theta, \mathcal{G},\mathbb{Q}_\phi)$ such that 
\begin{align}
    \Psi_t = \Phi_t(\phi)\ \ \ \ \ \mathbb{Q}_\phi\mbox{-a.s.},\ \mu^{(\nu)}\mbox{-a.e.}\ \phi,  \label{eqdirichlet2}
\end{align}
where $\Phi$ is the strong solution of \eqref{eq1} driven by $\mathcal{W}$. More precisely, we define $\Phi_t(\phi)\coloneqq \lim_{N\rightarrow\infty}\Phi^N_t(\phi)$ noting that the limit exists in $C([0,T];H^{-\beta}(\Lambda))$ $\mathbb{Q}_\phi$-almost surely for $\mu^{(\nu)}$-almost every $\phi \in E$, where $\Phi^N(\phi)$ is the solution to
\begin{align*}
\begin{cases}
\partial_t \Phi^N_t = \frac{1}{2} (\Delta - 1)\Phi^N_t - \frac{1}{2} \int_{[-\alpha_0, \alpha_0]}\alpha\exp_N^{\Diamond}{(\alpha \Phi^N_t)}\nu(d\alpha) + P_N\dot{\mathcal{W}}_t \label{eq1} \\
\Phi^N_0 = P_N \phi,
\end{cases}
\end{align*}
see Remark \ref{rem:thm1}.
Therefore, from \eqref{eqdirichlet1} and \eqref{eqdirichlet2}, we get
\begin{align*}
P_t F(\phi) = \mathbb{E}^{\mathbb{Q}_\phi}[F(\Phi_t (\phi)) ]\ \ \ \ \ \mu^{(\nu)}\mbox{-a.e.}\ \phi\in E 
\end{align*}
for any bounded measurable function $F$ on $E$ and $t\ge 0$.

\ \\
\noindent \textbf{Acknowledgement.}
This work is supported by JSPS KAKENHI Grant Numbers 24KJ1329, 19H00643, 22H00099, 23K03155, and 23K20801.

\bibliographystyle{alpha}

\bibliography{weighted_exp}

\end{document}